\documentclass[a4paper, 12pt]{amsart}

\newcommand{\myl}{\ell}

\usepackage{amsmath}
\usepackage{amssymb}
\usepackage[multiple]{footmisc}

\DeclareMathOperator{\ad}{ad}

\DeclareMathOperator{\grend}{GrEnd}
\DeclareMathOperator{\codim}{codim}
\DeclareMathOperator{\im}{Im}
\DeclareMathOperator{\End}{End}
\DeclareMathOperator{\frend}{\mathfrak{gl}}

\newcommand{\F}{\mathbb F}

\newcommand{\E}{\mathbb E}

\newcommand{\N}{\mathbb N}

\newcommand{\Set}[1]{\left\{ #1 \right\}}

\renewcommand{\phi}{\varphi}

\newtheorem{theorem}{Theorem}
\newtheorem{cor}[theorem]{Corollary}
\newtheorem{lemma}[theorem]{Lemma}
\newtheorem{prop}[theorem]{Proposition}

\theoremstyle{definition}

\newtheorem{example}[theorem]{Example}

\theoremstyle{remark}
\newtheorem{rem}[theorem]{Remark}

\begin{document}

\title[Thin subalgebras of Lie algebras of maximal class]
      {Thin subalgebras of\linebreak
        Lie algebras of maximal class}
      
      \author[Avitabile]{M.~Avitabile}
      
 \address[M.~Avitabile]{Dipartimento di Matematica e Applicazioni\\
 Universit\`a degli Studi di Milano - Bicocca\\
 via Cozzi, 55\\
 I-20125 Milano\\
 Italy}
 
 \email{marina.avitabile@unimib.it}

 \author[Caranti]{A.~Caranti}

 \address[A.~Caranti]{Dipartimento di Matematica\\
 Universit\`a degli Studi di Trento\\
 via Sommarive, 14\\
 I-38123 Trento\\
 Italy}

 \email{andrea.caranti@unitn.it}

 \author[Gavioli]{N.~Gavioli}

 \address[N.~Gavioli]{Dipartimento di Ingegneria e Scienze dell'Informazione e Matematica\\
 Universit\`a degli Studi dell'Aquila\\
 via Vetoio\\
 I-67100 Coppito (AQ)\\
 Italy}

 \email{norberto.gavioli@univaq.it}

 \author[Monti]{V.~Monti}
 
 \address[V.~Monti]{Dipartimento di Scienza e Alta Tecnologia\\
 Universit\`a degli Studi dell'Insubria\\
 Via Valleggio, 11\\
I-22100 Como\\
 Italy}

 \email{valerio.monti@uninsubria.it}

 \author[Newman]{M.~F.~Newman}

 \address[M.~F.~Newman]{Mathematical Sciences Institute\\
Australian National University \\
Canberra \\
ACT 2600 Australia}

 \email{mike.newman@anu.edu.au}

 \author[O'Brien]{E.~A.~O'Brien}

 \address[E.~A.~O'Brien]{Department of Mathematics\\
University of Auckland\\
Auckland \\
New Zealand}

 \email{e.obrien@auckland.ac.nz}

\renewcommand{\shortauthors}{M.~Avitabile et al}

\thanks{The first four authors are members of the Italian INdAM-GNSAGA, and
 thank the Department of Mathematics and its Applications of the
 University of Milano – Bicocca for the hospitality.
\endgraf
A.~Caranti is grateful to the Department of
Mathematics of the University of Trento for its support.
\endgraf
E.~A.~O'Brien has been supported by the Marsden Fund of New Zealand
grant UOA 1626}.

\begin{abstract}
  For every  field \(\F\)  which has a  quadratic extension  \(\E\) we
  show there  are non-metabelian infinite-dimensional thin  graded Lie
  algebras all of whose homogeneous components, except the second one,
  have   dimension   2.   We    construct   such   Lie   algebras   as
  \(\F\)-subalgebras  of  Lie algebras  \(M\)  of  maximal class  over
  \(\E\).  We  characterise the  thin Lie \(\F\)-subalgebras  of \(M\)
  generated  in degree  \(1\). Moreover  we show  that every  thin Lie
  algebra \(L\) whose  ring of graded endomorphisms of  degree zero of
  \(L^3\) is a  quadratic extension of \(\F\) can be  obtained in this
  way. We  also characterise  the 2-generator \(\F\)-subalgebras  of a
  Lie  algebra  of  maximal  class   over  \(\E\)  which  are  ideally
  \(r\)-constrained for a positive integer \(r\).
\end{abstract}

\maketitle

\thispagestyle{empty}

\section{Introduction}
A {\em thin} Lie algebra is a Lie algebra
\[
 L=\bigoplus_{i=1}^{\infty}L_{i},
\]
over  a field  \(\F\), graded  over the  positive integers,  such that
\(\dim_{\F}L_1=2\) and  the following \emph{covering  property} holds:
for  every  \(i\ge  1\)  and   for  every  nonzero  \(u  \in  L_{i}\),
\([u,L_{1}]=L_{i+1}\). Hence  each homogeneous component of  \(L\) has
dimension  at most~2.  If \(\dim  L_i=1\)  for every  \(i\ge 2\),  the
algebra \(L\) is referred to  as a {\em {\rm(}graded{\rm)} Lie algebra
  of maximal  class} (see Section~\ref{sec:prel} for  details). We use
the  convention that,  unless otherwise  specified, thin  algebras are
infinite-dimensional and  not of  maximal class. The  smallest integer
\(k\ge  2\) such  that \(\dim  L_k=2\) is  a parameter  that has  been
studied  in several  papers  (\cite{CMNS,AviJur,AJM}); in  particular,
\(k\) can be \(3\).

In  this paper  we  study the  class  of thin  algebras  all of  whose
homogeneous components,  except the second, have  dimension \(2\). The
only   known  examples   of   such  algebras   are  those   considered
in~\cite{GMY}.  There  it  is  proved that  all  metabelian  thin  Lie
algebras  belong   to  this   class,  and   they  are   in  one-to-one
correspondence with the  quadratic extensions of the  field \(\F\). In
particular,  in  one  of   the  constructions  in~\cite{GMY},  a  thin
metabelian Lie algebra is realised as  a subalgebra over \(\F\) of the
tensor  product  of  the unique  infinite-dimensional  metabelian  Lie
algebra of maximal class by a quadratic extension of \(\F\).

The main goal of this paper is to generalise the results of~\cite{GMY}
to the  non-metabelian case.  Let \(\E\) be  a quadratic  extension of
\(\F\) and let \(M\) be a Lie algebra of maximal class over \(\E\). In
Theorem~\ref{theo:main}  we consider  the \(\F\)-subalgebras  of \(M\)
generated by  two elements  of degree 1.  Amongst these  we explicitly
characterise those  which are  thin and  prove that  every homogeneous
component, except the  second one, has dimension \(2\).  We also prove
that,   for   every   such   thin   algebra   \(L\),   the   ring   of
\(L\)-endomorphisms  of  the   \(L\)-module  \(L^3\)  preserving  each
homogeneous component is isomorphic to \(\E\).

We consider more  generally the \(\F\)-subalgebras of  \(M\) which are
2-generated  in  degree 1  under  the  assumption  that \(\E\)  is  an
arbitrary extension of  \(\F\). In Proposition~\ref{prop:maximalclass}
we  characterise those  subalgebras which  are of  maximal class  over
\(\F\).   The   covering   property   for   a   graded   Lie   algebra
\(L=\bigoplus_{i=1}^\infty   L_{i}\)  can   be  restated   as  follows
(see~\cite{GM02}):  every nonzero  graded  ideal of  \(L\) is  located
between two consecutive terms of the  lower central series of \(L\). A
natural generalisation is to require  that there is a positive integer
\(r\) such that every graded nonzero ideal of \(L\) is located between
\(L^{i}\) and  \(L^{i+r}\) for some  \(i\). Lie algebras  generated by
\(L_{1}\)    satisfying     this    condition     are    \emph{ideally
  \(r\)-constrained}.   They   were  introduced   in~\cite{GM02}.   In
Proposition~\ref{prop:r-IC}  we characterise  the 2-generator  ideally
\(r\)-constrained    \(\F\)-subalgebras   of    \(M\)   in    positive
characteristic  under  the  assumption  that  \(\E\)  is  a  quadratic
extension of~\(\F\).

In    Section~\ref{sec:mxl_class}    we    consider    a    \emph{just
  infinite-dimensional}  algebra  \(T\)  over  \(\F\)  with  \(\dim_\F
T_i=2\) in each degree \(i\) different  from 2. We prove that the ring
\(\E\) of  \(T\)-endomorphisms of the \(T\)-module  \(T^3\) preserving
the homogeneous components is a field extension of \(\F\) of degree at
most 2.  When \(\E \neq  \F\), we construct  a Lie algebra  of maximal
class  over  \(\E\)  which  contains  \(T\)  as  an  \(\F\)-subalgebra
generated by two elements in degree~1.

Finally, we show  that if we apply this construction,  starting from a
thin \(\F\)-subalgebra of a Lie algebra  \(M\) of maximal class over a
quadratic  extension  \(\E\) of  \(\F\),  then  up to  isomorphism  we
recover \(M\). We deduce that  the number of thin Lie \(\F\)-algebras,
all  of  whose  homogeneous  components  of degree  at  least  3  have
dimension    2,   is    \(\lvert\F\vert^{\aleph_0}\),   the    maximum
possible. The  work of \cite{CaMa:thin,Young:thesis} provided  a lower
bound \(2^{\aleph_0}\)  for the  number of thin  Lie algebras  over an
arbitrary field of positive characteristic.

\section{Preliminaries}\label{sec:prel}

Throughout the  paper all  Lie algebras are  infinite-dimensional over
the underlying field, unless explicitly mentioned.

We briefly recall the definition and the basic notions of Lie algebras
of  maximal   class,  referring  the  reader   to~\cite{CMN,  CN}  for
details.  A  \emph{{\rm(}graded{\rm)}  Lie algebra  of  maximal  class
  {\rm(}generated in degree \(1\){\rm)}} is a Lie algebra
\[
 M=\bigoplus_{i=1}^{\infty}M_{i}
\]
over  a  field  \(\E\),  graded over  the  positive  integers,  having
\(\dim_{\E}  M_{1}=2\),   \(\dim_{\E}  M_{i}   =  1\)   otherwise  and
\([M_{i},M_{1}]=M_{i+1}\).  The last  condition is  equivalent to  the
requirement that \(M_1\) generates \(M\) as a Lie algebra. The name is
motivated  by the  observation  that the  quotient \(M/M^{j}\),  where
\(M^j=\bigoplus_{i\geq j}M_{i}\) is  the \(j\)-th Lie power,  is a Lie
algebra of (finite) dimension \(j\)  and nilpotency class \(j-1\), the
maximum possible value.

See  \cite{ShZe:narrow-Witt,  CV-L,  MR1968427} for  Lie  algebras  of
maximal class not generated in degree~\(1\).

The  {\em  \(2\)-step centralisers}  of  \(M\)  are the  1-dimensional
subspaces of \(M_{1}\) defined by
\begin{equation*}
 C_{i}=C_{M_{1}}(M_{i})=\{a \in M_{1} \mid [a,M_i]=0\},
\end{equation*}
for  \(i\geq  2\).  Note  that  \(C_{i}=C_{M_{1}}(u_{i})\)  for  every
\(0\neq u_{i} \in M_{i}\). Homogeneous generators, \(x\) and \(y\), of
\(M\) are chosen  as follows. Let \(y\in M_{1}\)  such that \(C_{2}=\E
y\). If all  other \(C_i\) coincide with \(C_{2}\), then  \(M\) is the
unique  metabelian Lie  algebra of  maximal  class (this  is the  only
possibility if the  underlying field has characteristic  zero), and we
choose   \(x\)  in   \(M_{1}\)   such  that   \(x\)   and  \(y\)   are
\(\E\)-linearly  independent.  Otherwise,  if \(i\)  is  the  smallest
integer such that \(C_{i}\neq C_{2}\),  then we choose \(x\) such that
\(C_{i}=\E x\). Such \(x\) and  \(y\) are {\em standard generators} of
\(M\); each is defined up to a nonzero constant in \(\E\).

The  sequence  of  centralisers   of  \(M\)  consists  of  consecutive
occurrences of \(\E  y\) interrupted by isolated  occurrences of other
centralisers. If  \(C\) is a  \(2\)-step centraliser and \(m\)  is the
smallest  integer  such that  \(C_{m}=C\),  then  \(m=2p^n\) for  some
non-negative  integer \(n\).  If  \(C_i\) and  \(C_j\) are  successive
occurrences  of \(C\),  then \(j-i  \le  m\) (see  \cite[Secs.\ 3  and
  10]{CN} and \cite[p.\ 439]{Jurman05}). Furthermore, as a consequence
of        the        \emph{specialisation       technique}        (see
\cite[Proposition~4.1]{CN}),    every   2-step    centraliser   occurs
infinitely often.

\section{From maximal class to thin}\label{sec:r-IC}

Let \(\E\) be an arbitrary extension of  \(\F\) and let \(M\) be a Lie
algebra  of maximal  class over  \(\E\). We  start with  the following
result whose proof is immediate.
\begin{lemma}\label{lemma:Fisomorphism}
 Let \(i\geq 2\) and \(\myl\in M_{1}\setminus C_{i}\). The map
 \[
  \begin{aligned}
        \ad \myl \colon M_{i} &\rightarrow M_{i+1}\\
          v &\mapsto [v,\myl]
  \end{aligned}
  \]
defines an  \(\E\)-isomorphism; it  is also an  \(\F\)-isomorphism. If
\(V\)  is  an  \(\F\)-subspace  of  \(M_{i}\),  then  \(\dim_{\F}  \ad
\myl(V)=\dim_{\F}[V,\myl]=\dim_{\F}V\).
\end{lemma}

We now consider the \(\F\)-subalgebra  \(L\) of \(M\) generated by two
elements  \(X\)  and  \(Y\)  of  \(M_{1}\). If  \(X\)  and  \(Y\)  are
\(\E\)-linearly   dependent   then   \([Y,X]=0\);   thus   \(L\)   has
\(\F\)-dimension  at most  \(2\) and  the Lie  product is  trivial. So
assume  that  \(X\) and  \(Y\)  are  \(\E\)-linearly independent.  Now
\([Y,X]\ne  0\),  so  \(\dim_{\F}L_{2}=1\).  Let \(l\)  be  a  nonzero
element of \(L_{i}\) for  \(i \geq 2\). Since \(C_{M_{1}}{(l)}=C_{i}\)
it   follows    that   \(C_{L_{1}}{(l)}=C_{i}\cap    L_{1}\).   Define
\(d_{i}=\dim_{\F}C_{i}\cap L_{1}\).

The following lemma is crucial for what follows.
\begin{lemma}\label{lem:centraliserinL}
  Let \(L\)  be the  Lie \(\F\)-subalgebra of  \(M\) generated  by two
  \(\E\)-linearly   independent   elements    \(X\)   and   \(Y\)   of
  \(M_{1}\). The following hold:
 \begin{enumerate}
 \item\label{dimensionseconddegree} \(\dim_{\F}L_{2}=1\);
 \item\label{nondecreasing} \(\dim_{\F}L_{i+1}\geq\dim_{\F}L_{i}\) for
   every \(i\geq 2\) (in particular, \(L\) is infinite-dimensional);
 \item\label{boundedcentraliser}  \(d_{i}\leq  1\) for  every  \(i\geq
   2\);
 \item\label{trivialcentraliser}       if       \(d_{i}=0\)       then
   \(\dim_{\F}[l,L_{1}]=2\) for every \(l\in L_{i}\setminus \{0\}\);
 \item\label{nontrivialcentraliser}      if      \(d_{i}=1\)      then
   \(\dim_{\F}[l,L_{1}]=1\) for every \(l\in L_{i}\setminus \{0\}\).
 \end{enumerate}
\end{lemma}

\begin{proof}
 Item~\eqref{dimensionseconddegree} is trivial. Since \(X\) and \(Y\)
 are \(\E\)-linearly independent, there is no integer \(i\) such that
 both \(X\) and \(Y\) belong to \(C_{i}\) (whose \(\E\)-dimension is
 \(1\)). This immediately yields \eqref{boundedcentraliser},
 and \eqref{nondecreasing} follows by
 \(L_{i+1}=[L_{i},X]+[L_{i},Y]\) and Lemma~\ref{lemma:Fisomorphism}.
 Items~\eqref{trivialcentraliser}
 and~\eqref{nontrivialcentraliser} are obvious.
\end{proof}

We  now characterise  the 2-generator  \(\F\)-subalgebras of  \(M\) of
maximal class.
\begin{prop}\label{prop:maximalclass}
 The algebra \(L\) has maximal class if and only if \(d_{i}=1\) for
 every \(i\geq 2\).
\end{prop}

\begin{proof}
 Observe \(L\) has  maximal class if and  only if \(\dim_{\F}L_{i}=1\)
 for     every    \(i\geq     2\).     Lemma~\ref{lem:centraliserinL},
 items~\eqref{dimensionseconddegree},    \eqref{nontrivialcentraliser}
 and~\eqref{trivialcentraliser} imply the claim.
\end{proof}

We now  restrict to quadratic  extensions. In  this case we  show that
\(L\)  is  thin  if  and  only  if  \(L_1\)  intersects  every  2-step
centraliser of \(M\) trivially.

\begin{theorem}\label{theo:main}
 If \(\vert \E:\F\vert=2\), then \(L\) is a thin Lie algebra if and
 only if \(d_{i}=0\) for every \(i\geq 2\). In this case
\(\dim_\F L_i=2\) for every \(i\ne 2\).
\end{theorem}

\begin{proof}
 If    \(l\)    is   a    nonzero    element    of   \(L_{1}\)    then
 \([l,L_{1}]=L_{2}\).   Thus   \(L\)   is   thin  if   and   only   if
 \([l,L_{1}]=L_{i+1}\) for  every \(i\geq 2\) and  every nonzero \(l\)
 \(\in\) \(L_{i}\) (and \(L\) is not of maximal class).

 Suppose   that    \(d_{i}=0\)   for    every   \(i\geq    2\).    Let
 \(0\not=\;\)\(l\)   \(\in\)   \(L_{i}\)    with   \(i\geq   2\):   by
 Lemma~\ref{lem:centraliserinL},      item~\eqref{trivialcentraliser},
 \(\dim_{\F}[l,L_{1}]=2\).         Since        \(\dim_{\F}L_{i+1}\leq
 \dim_{\F}M_{i+1}=2\), it follows that \([l,L_{1}]=L_{i+1}\) and \(L\)
 is thin. In particular, \(\dim_{\F}L_{i}=2\) for every \(i \neq 2\).

 Conversely, suppose  that \(L\)  is thin.  The covering  property and
 items~\eqref{trivialcentraliser} and \eqref{nontrivialcentraliser} of
 Lemma~\ref{lem:centraliserinL}      imply       that      \(\dim_{\F}
 L_{i+1}=2-d_{i}\),       for       every       \(i\ge2\).       Thus,
 Lemma~\ref{lem:centraliserinL},   item~\eqref{nondecreasing},  yields
 \(d_{i}\ge d_{i+1}\).  By \cite[Lemma~3.3]{CMN},  \(d_{i}=d_{2}\) for
 infinitely many  values of \(i\).  This implies that the  sequence of
 \(d_{i}\)s   is  constant.   By  Proposition~\ref{prop:maximalclass},
 \(d_{i}=0\) for every \(i\ge 2\).
\end{proof}

\begin{example}
  Suppose that \(M\) has  exactly two distinct \(2\)-step centralisers
  \(\E x \)  and \(\E y\). For every \(\gamma,  \delta \in \E\setminus
  \F\)  with \(\gamma\neq  \delta\),  the  \(\F\)-subalgebra of  \(M\)
  generated by \(x+y\) and \(\gamma x+\delta y\) is non-metabelian and
  thin and every homogeneous component  of degree different from \(2\)
  has dimension \(2\).
\end{example}

\begin{rem}\label{rem:alpha=1}
  Let \(\E=\F(\lambda)\) be a quadratic extension of \(\F\). Let \(x\)
  and \(y\) be standard generators of \(M\) and set \(X=\alpha x+\beta
  y\), \(Y=\gamma x+ \delta y\), where \(\alpha, \beta, \gamma, \delta
  \in  \E\)  are  such that  \(\alpha\delta-\beta\gamma\ne  0\).   But
  \({\E}   y\)   is   a    \(2\)-step   centraliser   of   \(M\),   so
  \(s\alpha+t\gamma  \neq  0\) for  every  \(s,t  \in \F\),  not  both
  zero. In  particular, \(\alpha, \gamma  \neq 0\).  Since  the linear
  map  that   multiplies  each  homogeneous  component   \(M_{i}\)  by
  \(\alpha^{-i}\)  is an  automorphism of  \(M\), we  can assume  that
  \(X=x+\beta y\) and  \(Y=\gamma x+\delta y\) with  \(\gamma \not \in
  \F\).  By possibly  adding to \(Y\) an \(\F\)-multiple  of \(X\) and
  multiplying   by   a   nonzero   element  of   \(\F\),   we   obtain
  \(\gamma=\lambda\).
	
  If  \(M\) is  not metabelian,  then \({\E}x\)  is also  a \(2\)-step
  centraliser of  \(M\), so  \(s\beta+t\delta\neq 0\) for  every \(s,t
  \in \F\), not  both zero.  It follows that  \(\beta\) and \(\delta\)
  are independent  over \(\F\), so  both are nonzero.  Since  \(y\) is
  defined up to  a nonzero constant in \(\E\), we  can take \(y'=\beta
  y\) in place of \(y\).  Under  this assumption, we can represent the
  generators of \(L\),  with respect to this new basis  of \(M_1\), as
  \(X=x+y\)   and  \(Y=\lambda   x+\delta  y\),   with  \(\delta   \in
  \E\setminus  \F\),  \(\delta   \neq  \lambda\).   Conversely,  these
  generate  a  thin  subalgebra  of \(M\)  under  the  assumptions  of
  Theorem~\ref{theo:main}.   Hence  there  exist at  most  \(\vert  \E
  \setminus  \F\vert\)  isomorphism  classes of  thin  subalgebras  of
  \(M\).
\end{rem}

\begin{rem}\label{rem:L^3}
  Let  \(L\) be  a thin  algebra as  in Theorem~\ref{theo:main}.  Then
  \(L_{i}=M_{i}\) as sets, for \(i\geq 3\), but we regard \(L_i\) as a
  vector space over \(\F\) and \(M_i\)  as a vector space over \(\E\).
  Consider    the    \(L\)-module     \(L^3\)    via    the    adjoint
  representation. Observe \(\E\)  acts by multiplication as  a ring of
  \(L\)-endomorphisms      of       \(L^3\),      preserving      each
  \(L_{i}=M_{i}\). Conversely,  let \(f\) be an  \(L\)-endomorphism of
  \(L^3\) preserving the homogeneous  components. For each \(i\geq 3\)
  take  a  nonzero   \(l_{i}\in  L_{i}\).  Then  \(f(l_{i})=\sigma_{i}
  l_{i}\) for some \(\sigma_{i}\in \E\). The covering property implies
  that \(l_{i+1}=[l_{i},l_{1}]\) for some \(l_{1}\in L_{1}\), hence
  \[
  \sigma_{i+1}l_{i+1}=f(l_{i+1})=f([l_{i},l_{1}])=[f(l_{i}),
    l_{1}]=\sigma_{i}[l_{i},l_{1}]=\sigma_{i}l_{i+1},
  \]
  and so \(\sigma_{i+1}=\sigma_{i}\).  In  particular, this shows that
  \(\sigma_{i}\)  is  independent  of   the  choice  of  \(l_{i}\)  in
  \(L_{i}\).  It follows that \(f\)  acts on \(L^3\) as multiplication
  by \(\sigma_{3}\in \E\).
\end{rem}

In characteristic  zero, the unique  Lie algebra of maximal  class has
just  one 2-step  centraliser, so  \(d_{i}=d_{2}\) for  every \({i\geq
  2}\).   Therefore  the   previous   results  settle   the  case   of
characteristic zero.

Recall from  the introduction that \(L\)  is ideally \(r\)-constrained
if, for every graded nonzero ideal  \(I\), there is a positive integer
\(i\) such that \(L^{i}\supseteq I\supseteq  L^{i+r}\). It is shown in
\cite[Proposition~2]{GM02}    that     a    \(2\)-generated    ideally
\(1\)-constrained Lie algebra  is either thin or of  maximal class. By
Proposition~\ref{prop:maximalclass}  and Theorem~\ref{theo:main}  this
happens if and only if the sequence \((d_{i})_{i\geq 2}\) is constant.

We  now address  the  case of  the subalgebras  \(L\)  for which  this
sequence is not constant.

\begin{prop}\label{prop:r-IC}
 Let \(\vert \E:\F\vert=2\). Assume that the sequence \((d_{i})_{i\geq
   2}\) is not constant and  set \(D_{0}=\{i\ge 2\mid d_{i}=0\}\). Let
 \(t_{1}<t_{2}<\dots<t_{j-1}<t_{j}<\cdots\)   be   the   elements   of
 \(D_{0}\). Then \(t_{j}-t_{j-1}\le t_{1}\)  for every \(j\ge 2\). Let
 \(r=\max(t_{j}-t_{j-1})\): then  \(2\le r\le t_{1}\) and  \(L\) is an
 ideally  \(r\)-constrained  Lie  algebra  but \(L\)  is  not  ideally
 \((r-1)\)-constrained. Moreover \(\dim_{\F}L_{i}=1\)  if \(2\le i \le
 t_{1}\) and \(\dim_{\F}L_{i}=2\) otherwise.
\end{prop}

\begin{proof}
  As \(\vert \E:\F\vert=2\), we have \(\dim_{\F}L_{i}\le 2\) for every
  \(i\ge2\).    Since   \(d_{t_{1}}\)   is   the   first   \(0\),   by
  Lemma~\ref{lem:centraliserinL}, items~\eqref{trivialcentraliser} and
  \eqref{nontrivialcentraliser}, \(\dim_{\F}L_{i}=1\) for every \(2\le
  i\le         t_{1}\)         and         \(\dim_{\F}L_{t_{1}+1}=2\).
  Lemma~\ref{lem:centraliserinL},    item~\eqref{nondecreasing}   then
  yields that  \(\dim_{\F}L_{i}\ge 2\) for every  \(i\ge t_{1}+1\) and
  equality     holds     since      \(\vert     \E:\F\vert=2\)     and
  \(\dim_{\E}L_{i}\le\dim_{\E}M_{i}=1\).

  As mentioned  in Section~\ref{sec:prel},  there are  infinitely many
  occurrences  of \(C_{t_{1}}\)  and if  \(C_u\) and  \(C_v\) are  two
  successive such  occurrences then \(v-u\le t_{1}\).  Hence \(D_{0}\)
  is  infinite   and  \(t_{j}-t_{j-1}\le  t_{1}\)  for   every  \(j\ge
  2\). Also  \(D_{1}=\{i\ge 2\mid  d_{i}=1\}\) is infinite,  so \(r\ge
  2\).

  Let \(I\) be a  nonzero graded ideal of \(L\) and  let \(i\) be such
  that \(I\subseteq L^{i}\) but \(I\nsubseteq L^{i+1}\). We claim that
  the ideal  generated by a nonzero  element \(l\) of \(I\)  of degree
  \(i\)  contains \(L^{i+r}\).  In particular,  \(I\supseteq L^{i+r}\)
  and therefore we need only consider the case when \(I\) is generated
  by \(l\). If  \(i=1\), then \([l,L_{1}]=L_{2}\) and we  are done. If
  \(2\le   i\le  t_{1}\)   then  \(l\)   generates  \(L_{i}\)   as  an
  \(\F\)-vector space and we are  done. So assume that \(i>t_{1}\) and
  let   \(t_{j}\)   be   such  that   \(t_{j}\ge   i>t_{j-1}\);   then
  \(t_{j}<i+r\)       by      definition       of      \(r\).       By
  Lemma~\ref{lem:centraliserinL}, items~\eqref{trivialcentraliser} and
  \eqref{nontrivialcentraliser}, we  have \(\dim_{\F}(I\cap L_{h})=1\)
  for \(i\le h\le t_{j}\)  and \(\dim_{\F}(I\cap L_{t_{j}+1})=2\) thus
  \(I\supseteq   L^{t_{j}+1}\supseteq   L^{i+r}\)  and   \(I\nsupseteq
  L^{t_{j}}\).     If     we     choose    \(j_{0}\)     such     that
  \(t_{j_{0}}-t_{j_{0}-1}=r\)  and  put \(i_{0}=t_{j_{0}-1}+1\),  then
  the  ideal generated  by an  element  of degree  \(i_{0}\) does  not
  contain    \(L^{i_{0}+r-1}\).   Hence    \(L\)   is    not   ideally
  \((r-1)\)-constrained.
\end{proof}

\section{From thin to maximal class}\label{sec:mxl_class}

Let \(L=\bigoplus_{i=1}^{\infty}L_i\) be a  graded Lie algebra, over a
field \(\F\), generated  by \(L_1\) as a Lie  algebra. A \emph{graded}
module  is an  \(L\)-module \(V=\bigoplus_{i=n_0}^{\infty}  V_i\) over
\(\F\),  for  some  \(n_0  \in   \N\),  such  that  \(V_i\cdot  L_j\le
V_{i+j}\). The module \(V\)  is \emph{just infinite-dimensional} if it
is infinite-dimensional  and every graded submodule  is either trivial
or has finite codimension. If  the image of the adjoint representation
of \(L\) is a just infinite-dimensional \(L\)-module, then we say that
\(L\) is a  just infinite-dimensional Lie algebra. This  is always the
case when  \(L\) is  an ideally  \(r\)-constrained Lie  algebra. Every
non-trivial    (possibly    non-graded)    submodule   of    a    just
infinite-dimensional graded  module \(V\) has finite  codimension (see
\cite[Lemma~2.3]{GMS2004}).   In   particular,  a   finite-dimensional
submodule of \(V\)  is necessarily trivial. A  linear map \(\phi\colon
V\to V\) is  {\em graded} of degree \(d\)  if \(\phi(V_i)\le V_{i+d}\)
for  a  non-negative integer  \(d\).  Let  \(\grend(V)\) be  the  ring
generated   by  the   graded   \(\F\)-linear  maps   of  \(V\).   Thus
\(\grend(V)=\bigoplus_{d=0}^{\infty} \grend_{d}(V)\) is a graded ring,
where  each homogeneous  component \(\grend_{d}(V)\)  consists of  the
\(\F\)-linear maps of \(V\) having degree \(d\).

\begin{lemma}[Schur's lemma]
 Let \(V\) be a graded just infinite-dimensional \(L\)-module. If
 \(0\ne \phi\in\End(V)\) is an \(L\)-endomorphism of \(V\), then \(\phi\)
 is an injective map.
\end{lemma}

\begin{proof}
  Let \(U=\ker \phi\)  and \(W=\im \phi \). By  assumption, \(W\ne 0\)
  so \(\dim U=\codim W\) is finite. Hence \(\ker \phi=0\).
\end{proof}

\begin{cor}\label{cor:just_infinite_module}
  Let \(V\)  be a  graded just infinite-dimensional  \(L\)-module. The
  ring  \(\grend_{L}(V)\), consisting  of  the \(L\)-endomorphisms  in
  \(\grend(V)\),  is  a  graded  domain  whose  homogeneous  component
  \(\grend_{L,d}(V)\) of degree \(d\) has \(\F\)-di\-men\-sion at most
  \(\inf_{i\ge  n_0}(\dim  V_{d+i})\).   In  particular,  the  subring
  \(\E=\grend_{L,0}(V)\) of the graded \(L\)-endomorphisms of \(V\) of
  degree \(0\) is  a skew-field extension of \(\F\) of  degree at most
  \(\inf_{i\ge n_0}(\dim V_i)\) and contains \(\F\) in its centre.
\end{cor}

\begin{proof}
  Let  \(v\) be  a nonzero  element in  \(V_i\). The  map \(\psi\colon
  \grend_{L,d}(V)\to V_{d+i}\)  defined by \(\phi\mapsto  \phi(v)\) is
  an  injective \(\F\)-linear  map, so  
  \begin{equation*}
    \dim_{\F} \grend_{L,d}(V)  =
    \dim_\F \im \psi \le \dim_{\F} V_{i+d}.
  \end{equation*}
\end{proof}

\begin{rem}\label{rem:field}
  For   every  \(i\)   the   homogeneous  component   \(V_i\)  is   an
  \(\E\)-module. In particular, \(\vert \E:\F \vert \cdot \dim_\E V_i=
  \dim_\F V_i\). If  \(\dim_\F V_i\) is a prime \(p\)  for some \(i\),
  then \(\vert \E:\F \vert \) is \(1\) or \(p\), so \(\E\) is a simple
  extension of the field \(\F\).
\end{rem}

\begin{rem}\label{rem:faithful}
  Let   \(I\)   be   a   non-trivial    graded   ideal   in   a   just
  infinite-dimensional Lie algebra  \(L\).  The centraliser \(C_L(I)\)
  is also  a graded  ideal of  \(L\), so  either \(C_L(I)=0\),  or the
  quotient  \(L/C_L(I)\) has  finite  dimension. In  the latter  case,
  \(C_L(I)\cap I\) is abelian and  has finite codimension. Hence \(L\)
  is soluble if and only if \(L\) has a non-trivial graded ideal \(I\)
  with  \(C_L(I)\ne 0\).   Also,  a  graded ideal  \(I\)  of \(L\)  is
  abelian if  and only if  \(C_L(I)\ne 0\) (see  \cite[Proposition 2.9
    and Corollary  2.11] {GMS2004}).  Therefore \(L\) is  insoluble if
  and only  if for every non-trivial  graded ideal \(I\) of  \(L\) the
  adjoint representation of \(L\) over \(I\) is faithful.
\end{rem}

From now  on, \(T\)  is a just  infinite-dimensional Lie  algebra over
\(\F\)  such  that   \(\dim_\F  T_i=2\)  for  every   \(i\ne  2\).  By
Corollary~\ref{cor:just_infinite_module}  and  Remark~\ref{rem:field},
the  ring  \(\E\) of  \(T\)-endomorphisms  of  \(T^3\) preserving  the
homogeneous components  is a  field extension of  \(\F\) of  degree at
most \(2\).
\begin{lemma}\label{lem:abelian}
  Suppose the field \(\E\) has degree \(2\) over \(\F\).
 If \(I\) is a non-trivial maximal abelian graded ideal of
 \(T\), then it equals \(T^k\) for some \(k\geq 2\) and is unique.
\end{lemma}

\begin{proof}
  Let \(k\)  be the smallest  integer such that \(I\cap  T_{k}\ne 0\).
  Note  \(k   \ge  2\)  because  \(\dim_\F(T_3)=2\).   We  claim  that
  \(I=T^k\): since  the ideal generated  by \(T_k\) is \(T^k\),  it is
  enough to  show that  \(I\cap T_{k}=T_{k}\).  When \(k=2\),  this is
  obvious  as  \(\dim_{\F}T_2=1\).   When   \(I\le  T^3\),  the  ideal
  \(\alpha I\)  is abelian for every  \(\alpha\in\E\).  As \([x,\alpha
    y]=\alpha[x,y]=0\) for every  \(x\) and \(y\) in  \(I\), the ideal
  \(I+\alpha  I\)  is  abelian  and \(\alpha  I\subseteq  I\)  by  the
  maximality of \(I\).  Therefore \(I\cap T_{k}\) is  a nonzero vector
  space over \(\E\), whence the claim.
\end{proof}

When  \(\vert   \E:\F  \vert  =   2\)  we  construct  a   Lie  algebra
\(N=\bigoplus_{i=1}^{\infty}N_{i}\)  of  maximal   class  over  \(\E\)
containing \(T\) as an \(\F\)-subalgebra  generated by two elements of
\(N_1\).

Assume first that \(T\)  is not metabelian. Let \(R=\frend_\E(\E\oplus
I)\) be  the Lie  algebra of  \(\E\)-linear maps  of the  vector space
\(\E\oplus I\), where \(I=T^3\) if \(T\) is insoluble, or \(I=T^k\) is
the unique maximal abelian graded ideal of \(T\), otherwise. Note that
\(k\ge 3\)  as \(T\) is not  metabelian, so \(I\) is  an \(\E\)-vector
space.  Let  \(z\in T\setminus I\)  be a homogeneous element  which is
central modulo~\(I\). By Remark~\ref{rem:faithful}, the map
\[
\rho\colon t\mapsto \left[
 \begin{array}{c|c}
 0 & [z,t]\\
 \hline
 0 & \ad t|_{I}
 \end{array}
 \right] \, ,
\]
which is the  adjoint representation of \(T\) over  the ideal \(\left<
z\right>\oplus  I\),   is  a  faithful  representation   of  \(T\)  in
\(R\). This  is easy  in the  insoluble case;  for soluble  \(T\), see
\cite[Lemma~2.8]{GMS2004}, taking \(J=\F z+I\).

The field  \(\E\) and  \(\rho(T)\) are  \(\F\)-subspaces of  \(R\). We
claim that
\[
N = \E\cdot \rho(T)=\{\lambda \cdot \rho(t) \mid \lambda\in \E \text{
 and } t\in T \}\subseteq R
\]
is a Lie algebra of maximal class over \(\E\). Let \(N_i=\E\cdot \rho(
T_i)\).  If  \(i\ge  3\) then  \(N_i=\E\cdot  \rho(T_i)=\rho  (\E\cdot
T_i)=\rho (T_i)\). Hence \(N_i\) is  a module over \(\E\) of dimension
\(1\)  for  every  \(i\ge  3\).  Clearly,  the  homogeneous  component
\(N_2=\E\cdot  \rho  (T_2)\)  has  dimension \(1\)  over  \(\E\).  The
dimension  of  \(N_1\)  over  \(\E\)  is  \(2\),  otherwise  \(T\)  is
abelian.  Let \(S\)  be the  \(\E\)-subalgebra of  \(N\) generated  by
\(N_1\).  This  algebra  contains   the  \(\F\)-algebra  generated  by
\(\rho(T_1)\), namely  \(\rho(T)\), so  \(S\) contains  also \(\E\cdot
\rho(T)=N\).  Hence  \(N\)  is  generated over  \(\E\)  by  its  first
homogeneous component \(N_1\).

When  \(T\)  is  metabelian,  the construction  must  be  modified  by
considering, instead of \(\rho\), the adjoint representation \(\rho'\)
over the ideal \(\left<Y\right>\oplus T^2\):
\[
\rho'\colon t\mapsto \left[
\begin{array}{c|c|c}
	0 & \alpha & [Y,s]\\
	\hline
	0 & 0 & [[Y,X], t]\\
	\hline
	0 & 0 & \ad t|_{T^3}
\end{array}
\right] \in R'=\frend(\E\oplus \E \oplus T^3) \] where \(T_1=\F\cdot X
+ \F\cdot  Y\) and \(t=\alpha  X+\beta Y +  s\) with \(s\in  T^2\). As
above, the Lie algebra \(N=\E\cdot \rho'(T)\) is of maximal class over
\(\E\).
\smallskip

Now let  \(\E\supset \F\) be a  quadratic extension of fields  and let
\(M\) be a Lie algebra of maximal  class over \(\E\). Let \(L\) be the
thin \(\F\)-subalgebra of \(M\) generated by two elements of \(M_{1}\)
as in  Theorem~\ref{theo:main}. By Remark~\ref{rem:L^3}, we  can apply
the above construction,  starting from the algebra \(L\),  to obtain a
Lie algebra \(N\) of maximal class  over \(\E\) containing \(L\) as an
\(\F\)-subalgebra. We show that we can recover the algebra \(M\) up to
isomorphism.
\begin{prop}\label{prop:back}
 The Lie algebras \(M\) and \(N\) of maximal class are isomorphic.
\end{prop}

\begin{proof}
  Observe  \(N=\E\cdot \rho(  L)=\rho  (\E  \cdot L)=\rho(M)\cong  M\)
  where  the   last  isomorphism  holds   because  \(M\)  is   a  just
  infinite-dimensional Lie algebra.
\end{proof}

\begin{rem}\label{rem:upperbound}
  Let \(M\) be an arbitrary algebra  of maximal class over \(\E\) with
  standard generators  \(x\) and \(y\).  Every  \(2\)-step centraliser
  other  than \(\E  y\) can  be written  as \(\E  (x+\lambda y)\)  for
  unique \(\lambda  \in \E\).  Let  \(\mathcal{L}\) be the set  of all
  such  \(\lambda\)s.   Let \(X=  \alpha  x  +  \beta  y =  \alpha(  x
  +\alpha^{-1}\beta  y)\) and  \(Y =  \gamma x+  \delta y=\gamma  (x +
  \gamma^{-1}\delta y)\)  be two \(\E\)-linearly  independent elements
  in  \(M_1\).  By  Theorem~\ref{theo:main}  the \(\F\)-subalgebra  of
  \(M\) generated by \(X\) and \(Y\) is thin if and only if \(\F \cdot
  X+\F \cdot  Y\) intersects  every \(2\)-step  centraliser trivially.
  This    occurs     if    and    only    if     the    affine    line
  \({\Lambda}=\Set{t\alpha^{-1}\beta+(1-t)\gamma^{-1}\delta\mid   t\in
    \F}\)  intersects  \(\mathcal{L}\)   trivially,  so  \(\mathcal{L}
  \subseteq    \E\setminus   \Lambda\).     There    are   at    least
  \(\vert\E\setminus
  \F\vert^{\aleph_0}=\vert\E\vert^{\aleph_0}=\vert\F\vert^{\aleph_0}\)
  algebras of  maximal class satisfying this  condition \cite{CMN}. By
  Proposition~\ref{prop:back},  there  are  at   least  as  many  thin
  \(\F\)-algebras  all of  whose homogeneous  components of  degree at
  least    3   have    dimension   2.     Hence   the    upper   bound
  \(\vert\F\vert^{\aleph_0}\)   for   the    number   of   such   thin
  \(\F\)-algebras is attained.
\end{rem}

\bibliographystyle{amsalpha}

\def\polhk#1{\setbox0=\hbox{#1}{\ooalign{\hidewidth
  \lower1.5ex\hbox{`}\hidewidth\crcr\unhbox0}}}
\providecommand{\bysame}{\leavevmode\hbox to3em{\hrulefill}\thinspace}
\providecommand{\MR}{\relax\ifhmode\unskip\space\fi MR }
\providecommand{\MRhref}[2]{%
  \href{http://www.ams.org/mathscinet-getitem?mr=#1}{#2}
}
\providecommand{\href}[2]{#2}

\end{document}